\documentclass[12pt]{amsart}
\usepackage{amsmath}	
\usepackage[dvips]{graphicx}
\input{amssymb.sty}
\input xy
\xyoption{all}
\theoremstyle{plain}
  \newtheorem{thm}{Theorem}[section]

  \newtheorem{fact}[thm]{Fact}
\theoremstyle{definition}

  \newtheorem{clm}[thm]{Claim}
\theoremstyle{remark}

\newcommand{\Amp}{{\operatorname{Amp}}}
\newcommand{\Aut}{\operatorname{Aut}}

\newcommand{\CC}{{\mathbb{C}}}

\newcommand{\Ext}{\operatorname{Ext}}

\newcommand{\NS}{\operatorname{NS}}

\newcommand{\Pic}{\operatorname{Pic}}

\newcommand{\ZZ}{{\mathbb{Z}}}

\begin{document}
\bibliographystyle{amsplain}
\title[Finiteness of isomorphic classes]
{Finiteness of isomorphic classes in the set of \\
moduli schemes of sheaves on a surface }
\author{Kimiko Yamada}
\email{kyamada@math.kyoto-u.ac.jp}
\address{Department of mathematics, Kyoto University, Japan}
\maketitle
\begin{abstract}
When a non-singular complex projective surface $X$ satisfies that $K_X\sim 0$, 
we shall show that there are only finitely many isomorphic classes as
abstract schemes in the set of moduli schemes of $H$-semistable sheaves with
fixed Chern classes $\alpha$ on $X$, where $H$ runs over the set of all $\alpha$-generic
polarizations on $X$.
\end{abstract}
\section{Introduction}
Let $X$ be a non-singular projective surface over $\CC$, 
$\alpha$ an element of $(r,c_1,c_2)\in \ZZ_{>0}\times \NS(X)\times \ZZ$ with $r>0$, 
and $H$ an ample line bundle (polarization) on $X$.
Then there is a coarse moduli scheme $M(H,\alpha)$ of $H$-semistable
sheaves of type $\alpha$ on $X$. It is projective over $\CC$.
Here we say that $M(H,\alpha)$ and $M(H', \alpha)$ are
{\itshape isomorphic by definition} 
if (i) any sheaf of type $\alpha$ on $X$ is $H$-stable
(resp. $H$-semistable) if and only if it is $H'$-stable
(resp. $H'$-semistable), and (ii) any $H$-semistable two sheaves of type $\alpha$ 
are S-equivalent with respect to $H$-semistability if and only if
they are S-equivalent with respect to $H'$-semistability.
The set 
\[ \left\{ M(H,\alpha) \bigm| \text{$H$: $\alpha$-generic polarization on $X$} \right\}
\left/ \text{(isomorphism by definition)} \right. \]
is countable and can be infinite. However we shall prove
\begin{thm}\label{thm:main}
When $r>0$ and $K_X\sim 0$, the set
\[ \left\{ M(H,\alpha) \bigm| \text{$H$: $\alpha$-generic polarization on $X$}\right\}
 \left/ \text{(isomorphism as abstract schemes)} \right. \]
is finite. 
\end{thm}
From \cite{Qi:birational},
$M(H, \alpha)$ and $M(H', \alpha)$ are birationally equivalent if
$2rc_2-(r-1)c_1^2$ is sufficiently large with respect to $H$ and $H'$.
When $X$ is minimal and $\kappa(X)>0$
there is a moduli-theoretic analogy of minimal model program of
$M(H, \alpha)$ by \cite{Yam:flip}. If $K_X\sim 0$, then we can use Theorem
and $M(H, \alpha)$ and $M(H', \alpha)$ are connected by Mukai flops.
Thus when $X$ is minimal and $\kappa(X)\geq 0$, one can use there results
about birational relation between $M(H, \alpha)$ and $M(H',\alpha)$.\par
 The author expresses hearty thanks to Prof. S. Mukai and Prof. Y. Namikawa for their
invaluable suggestions and comments. 

\section{Proof of Theorem}
We begin with preliminary.
The notion of $\alpha$-{\it walls} (or walls with respect to $\alpha$)
appeared in \cite{EG:variation}, \cite{FQ:flips} and \cite{MW:Mumford}.
A connected component of the complement of the union of all $\alpha$-walls
in the ample cone $\Amp(X)$ of $X$ is called an $\alpha$-{\it chamber}.
A polarization on $X$ is said to be $\alpha$-{\it generic} if
it is contained in no $\alpha$-wall.\par
Now let $H$ and $H'$ be $\alpha$-generic polarizations.
As mentioned in Introduction, the set 
\begin{equation}\label{eq:isomdef}
 \left\{ M(H,\alpha) \bigm| \text{$H$: $\alpha$-generic polarization on $X$} \right\}
\left/ \text{(isomorphism by definition)} \right. 
\end{equation}
is countable and can be infinite. 
Indeed, in case where $r=2$, any $\alpha$-wall $W$ equals
$W^{\eta}=\left\{ L\in\Amp(X) \bigm| L\cdot \eta=0 \right\}$
with $\eta\in\NS(X)$ such that $0< -\eta^2 \leq 4c_2-c_1^2$.
When $\alpha$-generic polarizations $H$ and $H'$ lie in adjacent $\alpha$-chambers
separated by $W$, a sheaf $E$ of type $\alpha$ is $H$-semistable and not 
$H'$-semistable if and only if $E$ is given by a nonsplit extension
\begin{equation}\label{eq:ext}
 0 \longrightarrow {\mathcal O}_X(D)\otimes I_{Z_l} \longrightarrow
   E \longrightarrow {\mathcal O}_X(c_1-D)\otimes I_{Z_r} \longrightarrow 0, 
\end{equation}
where $Z_l$ and $Z_r$ are zero-dimensional subscheme in $X$ 
such that $l(Z_l)+l(Z_r)=c_2-c_1^2/4+ \eta^2/4$ and
$D$ is a divisor such that $2D-c_1 \sim \eta$.
When $K_X\sim 0$, it holds that
\begin{multline}\label{eq:chi}
 -\chi({\mathcal O}_X(c_1-D)\otimes I_{Z_r}, {\mathcal O}_X(D)\otimes I_{Z_l})=
 -\eta^2/2 +l(Z_l) +l(Z_r) -\chi({\mathcal O}_X)    \\
= 2c_2-c_1^2/2 -l(Z_l)-l(Z_r) -\chi({\mathcal O}_X) 
> 2c_2-c_1^2/2 -c_2+c_1^2/4 -\chi({\mathcal O}_X) > -\chi({\mathcal O}_X),
\end{multline}
where the first inequality holds since $-\eta^2/4= c_2-c_1^2/4-l(Z_l)-l(Z_r)>0$
by the Hodge index theorem,
and the second inequality holds from Bogomolov's inequality.
Thus if $K_X\sim 0$ and $\chi({\mathcal O}_X)\leq 0$, then
\eqref{eq:chi} implies 
$\Ext^1({\mathcal O}_X(c_1-D)\otimes I_{Z_r}, {\mathcal O}_X(D)\otimes I_{Z_l})\neq 0$
and a $H$-semistable sheaf of type $\alpha$ which is not $H'$-semistable
does exist. As a result $M(H,\alpha)$ and $M(H', \alpha)$ are not isomorphic 
by definition if
$\alpha$-generic polarizations $H$ and $H'$ are separated by a $\alpha$-wall.
On the other hand, Matsuki-Wentworth \cite{MW:Mumford} gave an example
of an Abelian surface 
where there are infinitely many $\alpha$-walls for some $\alpha=(2,c_1,c_2)$.
Consequently the set \eqref{eq:isomdef} can be infinite.\par
Now we prove Theorem \ref{thm:main}. We shall use the following three facts.
\begin{fact}\label{fact:otimesL}$[$\cite{Yos:abelian}, Lemma 1.1$]$
Let $L$ be a line bundle on $X$. When a polarization $H$ is $\alpha$-generic,
the map $E\mapsto E\otimes L$ induces an isomorphism
$\otimes L:\, M(H,\alpha) \rightarrow M(H, \alpha(L)) $, where $\alpha(L)$ is
the Chern class of $E\otimes L$ for a sheaf $E$ of type $\alpha$.
\end{fact}
\begin{fact}\label{fact:FiniteMeetLambda}$[$\cite{MW:Mumford}, Lemma 1.5'$]$
If $\Lambda\subset \overline{\Amp}(X)$ is a finite rational cone,
then only finitely many $\alpha$-walls intersect1 with $\Lambda$.
\end{fact}
\begin{fact}\label{fact:Morrison}$[$\cite{Sterk:finiteK3}, \cite{Kaw:coneCY}$]$
When $K_X\sim 0$, there is a finite rational cone 
$\Lambda\subset\overline{\Amp}(X)$ such that $\Aut(X)\cdot\Lambda \supset \Amp(X)$.
\end{fact}
In general $c_1(\tau_*(E))\neq c_1(E)$ in $\NS(X)$ for $\tau\in\Aut(X)$,
so $\tau$ does NOT induce an isomorphism between $M(H, \alpha)$ and 
$M(\tau_*(H),\alpha)$.
Thus Theorem \ref{thm:main} does not follow Fact \ref{fact:FiniteMeetLambda} and
Fact \ref{fact:Morrison} themselves.
For $\alpha=(r,c_1,c_2) \in \ZZ_{>0}\times\NS(X)\times\ZZ$,
put $\Aut'(X,c_1)=\left\{ \tau\in\Aut(X) \bigm| 
c_1\equiv\tau_* c_1 \text{ mod } r\NS(X)\right\}$.
Then $\Aut(X)\left/ \Aut'(X,c_1) \right.$ is finite since the map
$\iota:\, \Aut(X)/ \Aut'(X,c_1) \rightarrow \NS(X)/ r\NS(X)$
defined by $\iota(\tau\Aut'(X,c_1))=[c_1-\tau_*c_1]$ is injective and
$r>0$.
Let $\tau_1,\dots,\tau_N$ be representative elements of $\Aut(X)/ \Aut'(X,c_1)$.
By Fact \ref{fact:Morrison}, there is a finite rational cone
$\Lambda\subset\overline{\Amp}(X)$ such that $\Aut(X)\cdot\Lambda \supset \Amp(X)$.
Let $\Lambda'\subset\overline{\Amp}(X)$ be the finite rational cone spanned
by $\tau_1^*(\Lambda),\dots, \tau_N^*(\Lambda)$.
By Fact \ref{fact:FiniteMeetLambda}, only finitely many $\alpha$-chambers,
say ${\mathcal D}_1,\dots, {\mathcal D}_M$, intersect with $\Lambda'$.
Let $H_j$ be a polarization contained in ${\mathcal D}_j$.
\begin{clm}
For any $\alpha$-generic polarization $H$, we have some $1\leq j\leq M$
such that $M(H)\simeq M(H_j)$ as abstract schemes.
\end{clm}
\begin{proof}
 Since $\Aut(X)\cdot\Lambda\supset\overline{\Amp}(X)$, some $\tau\in\Aut(X)$
 satisfies $\tau_*H\in\Lambda$.
For some $1\leq i\leq N$, $[\tau]=[\tau_i]$ in $\Aut(X)/\Aut'(X,c_1)$, so
one has $\tau=\tau_i\cdot\tau_0$ with $\tau_0\in\Aut'(X,c_1)$.
Then $\tau_*H\in\Lambda$ implies that 
$\tau_{0*}H\in \tau_i^*\Lambda\subset \Lambda'$.
The map $E\mapsto \tau_{0*}E$ deduces an isomorphism
\[ \tau_{0*}:\, M(H,\alpha) \simeq M(\tau_{0*}H, (r, \tau_{0*}c_1, c_2)).\]
Because $\tau_0\in\Aut'(X,c_1)$, some $L\in\Pic(X)$ satisfies
$\tau_{0*}c_1=c_1-rL$ in $\NS(X)$.
By Fact \ref{fact:otimesL}, the map $E\mapsto E\otimes L$ deduces
\[ \otimes L:\, M(\tau_{0*}H, (r,\tau_{0*}c_1,c_2)) \simeq 
   M((r,\tau_{0*}c_1+rL=c_1,c_2), \tau_{0*}H). \]
(One can check that $c_2(\tau_{0*}E\otimes L)=c_2$.)
Since $\tau_{0*}c_1=c_1-rL$, one can verify that $\tau_{0*}H$ is $\alpha$-generic.
At last, $M(\tau_{0*}H, \alpha)=M(H_j, \alpha)$ for some $1\leq j\leq M$
because $\tau_{0*}H\in\Lambda'$.
\end{proof}
This claim ends the proof of Theorem \ref{thm:main}.
%

\providecommand{\bysame}{\leavevmode\hbox to3em{\hrulefill}\thinspace}
\providecommand{\MR}{\relax\ifhmode\unskip\space\fi MR }
\providecommand{\MRhref}[2]{%
  \href{http://www.ams.org/mathscinet-getitem?mr=#1}{#2}
}
\providecommand{\href}[2]{#2}

\end{document}